\documentclass[11pt]{article}
\usepackage{amsfonts,amsmath,amssymb,amsthm, amscd, subcaption}
\usepackage{graphicx}
\numberwithin{equation}{section}

\newtheorem{theorem}{Theorem}[section]
\newtheorem{prop}[theorem]{Proposition}
\newtheorem{lemma}[theorem]{Lemma}
\newtheorem{cor}[theorem]{Corollary}

\theoremstyle{definition}
\newtheorem{definition}[theorem]{Definition}
\newtheorem{example}[theorem]{Example}
\newtheorem{remark}[theorem]{Remark}

\def\<{{\langle}}
\def\>{{\rangle}}
\def\G{{\Gamma}}

\def\g{{\gamma}}
\def\d{{\delta}}
\def\Z{\mathbb Z}
\def\Q{\mathbb Q}
\def\R{\mathbb R}
\def\T{{\mathbb T}}
\def\S{{\mathbb S}}
\def\N{{\mathbb N}}

\def\Co{\mathbb C}

\def\l{\lambda}
\def\si{\sigma}
\def\t{\tau}

\def\w{\omega}
\def\Gr{\Bbb G}

\def\La{\Lambda}

\def\Rd{{\cal R}_d}
\def\s{{\bf s}}

\def\De{{\Delta}}
\def\o{\overline}

\def\ni{\noindent} 

\begin{document}

\title{Spanning Trees and Mahler Measure}

\author{Daniel S. Silver 
\and Susan G. Williams\thanks {The authors are partially supported by the Simons Foundation.} }

\maketitle 


\begin{abstract} The \emph{complexity} of a finite connected graph is its number of spanning trees; for a non-connected graph it is the product of complexities of its connected components. If $G$ is an infinite graph with cofinite free $\Z^d$-symmetry, then the logarithmic Mahler measure $m(\De)$ of its Laplacian polynomial $\De$ is the exponential growth rate of the complexity of finite quotients of $G$. It is bounded below by $m(\De(\Gr_d))$, where $\Gr_d$ is the grid graph of dimension $d$. The growth rates $m(\De(\Gr_d))$ are asymptotic to $\log 2d$ as $d$ tends to infinity. If $m(\De(G))\ne 0$, then $m(\De(G)) \ge \log 2$.  \bigskip

\end{abstract}

\section{Introduction.} \label{Intro} Efforts to enumerate spanning trees of finite graphs can be traced back at least as far as 1860, when Carl Wilhelm  Borchardt used determinants to prove that $n^{n-2}$ is the  number of spanning trees in a complete graph on $n$ vertices.\footnote{The formula is  attributed to Arthur Cayley, who wrote about the formula, crediting Borchardt, in 1889.}  The number of spanning trees of a graph, denoted here by $\t(G)$, is often called the \emph{complexity} of $G$. 

When the graph $G$ is infinite one can look for a sequence of  finite graphs $G_j, \ j \in \N$, that approximate $G$. Denoting by $|V(G_j)|$ the number of vertices of $G_j$,   a measure of asymptotic complexity for $G$ is provided by the limit:$$\displaystyle \limsup_{j \to \infty} \frac{1}{|V(G_j)|} \log \t(G_j).$$

Computing such limits has been the goal of many papers  (\cite{BP93, Gr76, Ke12, Ly05, SW00, Wu77} are just a few notable examples). Combinatorics combined with analysis are the customary tools. However, the integral formulas found are familiar also to those who work with algebraic dynamical systems \cite{LSW90,  Sc95}.

When the graph $G$ admits a cofinite free $\Z^d$-symmetry (see definition below), a precise connection with algebraic dynamics was made in \cite{LSW14}. For such graphs a finitely generated ``coloring module" over the ring of Laurent polynomials $\Z[x_1^{\pm 1}, \ldots, x_d^{\pm 1}]$ is defined. It is presented by a square matrix with nonzero determinant $\De(G)$. The polynomial $\De(G)$ has appeared previously (see \cite{Ly05}). The logarithmic Mahler measure $m(\De(G))$ arises now as the topological entropy of the corresponding $\Z^d$-action on the Pontryagin dual of the coloring module. The main significance for us is that $m(\De(G))$ determines the asymptotic complexity of $G$. This characterization was previously shown for connected graphs, first by R. Solomyak \cite{So98} in the case where the vertex set is $\Z^d$ and then for more general vertex sets  by R. Lyons \cite{Ly05}.

We present a number of results, many of them new,  about asymptotic complexity from the perspective of algebraic dynamics and Mahler measure. Where possible we review the relevant ideas.  \\

\ni {\bf Acknowledgements.} It is the authors' pleasure to thank Abhijit Champanerkar, Matilde Lalin and Chris Smythe for helpful comments and suggestions. 

\section{Spanning trees of finite graphs.} 
\begin{definition} \label{complexity} Let $G$ be a finite graph. We denote by $\t(G)$ the number of spanning trees of $G$. When $G$ is connected, $\t(G)$ is often called the \emph{complexity} of $G$. For a finite graph $G$ with connected components $G_1, \ldots, G_\mu$, we define the {\emph complexity} $T(G)$ to be the product $\t(G_1) \cdots \t(G_\mu)$.  \end{definition}

Upper bounds for $\t(G)$ are known. For example, there is the following theorem of \cite{Gr76}.

\begin{theorem}\label{upperbound} If $G=(V,E)$ is a finite connected graph with vertex and edge sets $V$ and $E$, respectively, then $$\t(G) \le \bigg( \frac{2|E|-\d}{|V|-1}\bigg)^{|V|-1},$$
where $\d$ is the maximum degree of $G$.  \end{theorem}

The complexity of a finite graph $G$ can be computed recursively using deletion and contraction of edges. The following is well known. A short proof can be found, for example, on page 282 of \cite{GR01}. 

\begin{prop}\label{recursive} If $G$ is a finite connected graph and $e$ is a non-loop edge, then $$\t(G) = \t(G\setminus e) + \t(G\big / e).$$ 
 \end{prop}
 
It is obvious that if $G$ is connected but $G\setminus e$ is not, then $\t(G) = T(G\setminus e)$. It follows that deleting or contracting edges of a  graph $G$ cannot increase the complexity $T(G)$. We will make frequent use of this fact here. 

\begin{definition} the \emph{Laplacian matrix} $L$ of a finite graph $G$ is the difference $D - A$, where 
$D$ is the diagonal matrix of degrees of $G$, and $A$ is the adjacency matrix of $G$, with $A_{i,j}$ equal to the number of edges between the $i$th and $j$th vertices of $G$. Loops in $G$ are ignored. 
\end{definition}

\begin{theorem}(Kirchhoff's Matrix Tree Theorem) If $G$ is a finite graph, then 
$\t(G)$ is equal to any cofactor of its Laplacian matrix $L$. \end{theorem} 

\begin{cor}\label{product}(see, for example, \cite{GR01}, p. 284) Assume that $G=(V,E)$ is a finite graph with connected components $G_1, \ldots, G_\mu$ and corresponding vertex sets $V_1, \ldots, V_\mu$. Then $$T(G) = \frac{1}{|V_1|\cdots |V_\mu|}\prod_\l \l,$$
where the product is taken over the set of nonzero eigenvalues of $L$. \end{cor} 

Useful lower bounds for $\t(G)$ are more rare. We have the following result of Alon. 

\begin{theorem}\label{lowerbound}\cite{Al90}   If $G=(V,E)$ is a finite connected $\rho$-regular graph, then 
$$\t(G) \ge [\rho(1-\epsilon(\rho))]^{|V|},$$ 
where $\epsilon(\rho)$ is a nonnegative function with $\epsilon(\rho)\to\infty$ as $\rho\to\infty$.
\end{theorem}


%
%

\section{Graphs with free $\Z^d$-symmetry and statement of results.} We regard $\Z^d$ as the multiplicative abelian group freely generated by $x_1, \ldots, x_d$. We denote the Laurent polynomial ring $\Z[\Z^d] = \Z[x_1^{\pm 1}, \ldots, x_d^{\pm 1}]$ by $\Rd$. As an abelian group $\Rd$ is generated freely by monomials $x^\s = x_1^{s_1} \ldots x_d^{s_d}$, where $\s = (s_1, \cdots, s_d) \in \Z^d$. 

Let $G=(V,E)$ be graph with a \emph{cofinite free $\Z^d$-symmetry}. By this we mean that $G$ has a free $\Z^d$-action by automorphisms such that the quotient  graph $\o G= (\o E, \o V)$ is finite. Such a graph is necessarily locally finite. The vertex set $V$ and the edge set $E$ consist of finitely many orbits $v_{1, \s}, \ldots, v_{n,\s}$ and $e_{1, \s}, \ldots, e_{m, \s}$, respectively. The $\Z^d$-action is determined by 
\begin{equation} x^{\s'} \cdot v_{i, \s} = v_{i, \s+ \s'}, \quad \quad   x^{\s'} \cdot e_{j, \s} = e_{j, \s+ \s'},\end{equation}
where $1\le i \le n,\ 1 \le j \le m$ and $\s, \s' \in \Z^d$. (When $G$ is embedded in some Euclidean space with $\Z^d$ acting by translation, it is usually called a \emph{lattice graph}. Such graphs arise naturally in physics, and they have been studied extensively.)

It is helpful to think of $G$ as a covering of a graph $\overline{G}$ in the $d$-torus $\T^d = \R^d/\Z^d$ (not necessarily embedded), with projection map determined by $v_{i, \s} \mapsto v_i$ and $e_{j, \s} \mapsto e_j$. The cardinality  $|\o V|$ is equal to the number $n$ of vertex orbits of $G$, while $|\o E|$ is the number $m$ of edge orbits. 

If $\La \subset \Z^d$ is a subgroup, then the intermediate covering graph in $\Bbb R^d/\La$ will be denoted by $G_\La$.  The subgroups $\La$ that we will consider have index $r < \infty$, and hence $G_\La$ will be a finite $r$-sheeted cover of $\o G$ in the $d$-dimensional torus  $\Bbb R^d/\La$.

Given a graph $G$ with cofinite free $\Z^d$-symmetry, the Laplacian matrix is defined to be the $(n \times n)$-matrix $L = D - A$, where now $D$ is the diagonal matrix of degrees of $v_{1, \s}, \ldots, v_{n, \s}$ while $A_{i, j}$ is the sum of monomials $x^\s$ for each edge in $G$ from $v_{i, \bf{0}}$ to $v_{j, \s}$.  The \emph{Laplacian polynomial} $\De$ is the determinant of $L$. It is well defined up to multiplication by units of the ring $\Rd$. Examples can be found in \cite{LSW14}.

The following is a consequence of the main theorem of \cite{Fo93}. It is made explicit in Theorem 5.2 of \cite{Ke12}. 

\begin{prop}\label{poly} \cite{Ke12} Let $G$ a graph with cofinite free $\Z^d$-symmetry. Its Laplacian polynomial has the form \begin{equation}\label{polys} \De(G) = \sum_F\  \prod_{\rm Cycles\ of\ F} (2-w -w^{-1}),\end{equation} where  the sum is over all cycle-rooted spanning forests $F$ of $\o G$, and $w, w^{-1}$ are the monodromies of the two orientations of the cycle. \end{prop} 

A cycle-rooted spanning forest (CRSF) of $\o G$  is a subgraph of $G$ containing all of $V$ such that each connected component has exactly as many vertices as edges and therefore has a unique cycle. The element $w$ is the monodromy of the cycle, or equivalently, its homology in $H_1(\T^d; \Z) \cong \Z^d$.
See \cite{Ke12} for details. 

A graph with cofinite free $\Z^d$-symmetry need not be connected. In fact, it can have countably many connected components. Nevertheless, the number of  $\Z^d$-orbits of components, henceforth called \emph{component orbits}, is necessarily finite. 

\begin{prop} \label{components}  If $G$ is a graph with cofinite free $\Z^d$-symmetry and component orbits $G_1, \ldots, G_t$, then $\De(G) = \De(G_1)\cdots \De(G_t)$. \end{prop}

\begin{proof} After suitable relabeling, the Laplacian matrix for $G$ is a  block diagonal matrix with diagonal blocks equal to the 
Laplacian matrices for $G_1, \ldots, G_s$. The result follows immediately. 
\end{proof}

\begin{prop} \label{zeropoly} Let $G$ a graph with cofinite free $\Z^d$-symmetry. Its Laplacian polynomial $\De$ is identically zero if and only $G$ contains a closed component. 
\end{prop} 

\begin{proof} If $G$ contains a closed component, then some component orbit $G_i$ consists of closed components.  We have $\De(G_i)=0$ by \ref{polys}, since all cycles of $\overline{G_i}$ have monodromy 0. By Proposition \ref{components}, $\De$ will be identically zero. 

Conversely, assume that no component of $G$ is closed. Each component of $\o G$ must contain a cycle with nontrivial monodromy. We can extend this collection of cycles to a cycle rooted spanning forest $F$ with no additional cycles.  The corresponding summand  in \ref{polys} has positive constant coefficient. Since every summand has nonnegative constant coefficient,  $\De$ is not identically zero.

\end{proof}

\begin{definition} \label{mahler} The \emph{logarithmic Mahler measure} of a nonzero polynomial 
$f(x_1, \ldots, x_d) \in \Rd$ is 
\begin{equation*} m(f) = \int_0^1 \ldots \int_0^1 \log|f(e^{2\pi i \theta_1}, \ldots, e^{2\pi i \theta_d})| d\theta_1 \cdots d\theta_d. \end{equation*}

\end{definition} 

\begin{remark} (1) The integral in Definition \ref{mahler} can be singular, but nevertheless it converges. 
(See \cite{EW99} for two different proofs.) \smallskip

(2) If $u_1, \ldots, u_d$ is another basis for $\Z^d$, then $\De(u_1, \ldots, u_d)$ has the same logarithmic Mahler measure as $\De(x_1, \ldots, x_d)$. \smallskip

(3) If $f, g \in \Rd$, then $m(fg) = m(f) + m(g)$. Moreover, $m(f) =0$ if and only if $f$ is a unit or a unit times a product of 1-variable cyclotomic polynomials, each evaluated at a monomial of $\Rd$
(see \cite{Sc95}). In particular, the Mahler measure of the Laplacian polynomial $\De$ is well defined.

(4) When $d=1$, Jensen's formula shows that $m(f)$ can be described another way. If $f(x) = c_s x^s+ \cdots c_1 x + c_0$, $c_0c_s \ne 0$, then
\begin{equation*} m(f) = \log|c_s| + \sum_{i=1}^s \log |\lambda_i|, \end{equation*}
where $\lambda_1, \ldots, \lambda_s$ are the roots of $f$.  \smallskip

\end{remark}

\begin{theorem}(cf. \cite{Ly05}) \label{limit} Let $G=(V,E)$ be graph with cofinite free $\Z^d$-symmetry. If $\De \ne 0$, then  
\begin{equation}  \lim_{\langle \La \rangle \to \infty}  \frac{1}{|\Z^d/\La|} \log T(G_\La)= m(\De), \end{equation}
where $\La$ ranges over all finite-index subgroups of $\Z^d$, and $\langle \La \rangle$  denotes the minimum length of a nonzero vector in $\La$. 
\end{theorem}

\begin{remark} (1) The condition $\langle \La \rangle \to \infty$ ensures that fundamental region of $\La$ grow in all directions. 

(2) In the case that $G$ is connected, each quotient $G_\La$ is also connected. In the statement of the theorem, $T(G_\La)$ is simply $\t(G_\La)$. In this case, Theorem \ref{limit} is proven in \cite{Ly05} for graphs of  greater generality.  

(3) Theorem \ref{limit} was established in \cite{LSW14} with the weaker limit superior rather than an ordinary limit. The stronger result will follow from analytical remarks in \cite{EW99} related to Mahler measure. \end{remark}

We call the limit in Theorem \ref{limit} the {\it complexity growth rate} of $G$, and denote it by $\gamma(G)$. Its relationship to the {\it thermodynamic limit} or {\it bulk limit} defined for a wide class of lattice graphs is discussed in  \cite{LSW14}. We briefly repeat the idea in order to state Corollary \ref{detgrowth}.

Denote by $R= R(\La)$ a fundamental domain of $\La$. Let $G\vert_R=(V_R, E_R)$ be the full subgraph of $G$ on vertices $v_{i, \s},\ \s \in R$. If  $G\vert_R$ is connected for each $R$, then by Theorem 7.10 of \cite{LSW14} the sequences $\{\tau(G_\La)\}$ and $\{\tau(G\vert_R)\}$ have the same exponential growth rates.   The bulk limit is then $\gamma(G)/|\overline V|$.

When $d \le 2$ and $G$ is a plane graph, the medial construction associates an alternating link diagram $\ell_R$ to $G\vert_R$, for any subgroup $\La \subset \Z^d$ and fundamental region region 
$R$. (This is illustrated in Figure \ref{gridlinks}. See \cite{Ka01} for details.) 

\begin{example} The $d$-dimensional grid graph $\Gr_d$ has vertex set $\Z^d$ and an edge 
from $(s_1, \ldots, s_d)$ to $(s_1', \ldots, s_d')$ if $|s_i -s_i'|=1$ and $s_j = s_j',\ j \ne i$, for every $1 \le i \le d$.  Its Laplacian polynomial is 
\begin{equation*} \De(\Gr_d) = 2d -x_1-x_1^{-1}- \cdots - x_d-x_d^{-1}. \end{equation*}
When $d=2$, it is a plane graph. The medial links $\ell_R$ are indicated in Figure \ref{gridlinks} for $\La = \langle x_1^2, x_2^2 \rangle$ on left and $\La= \langle x_1^3, x_2^3 \rangle$ on right.

\begin{figure}
\begin{center}
\includegraphics[height=2 in]{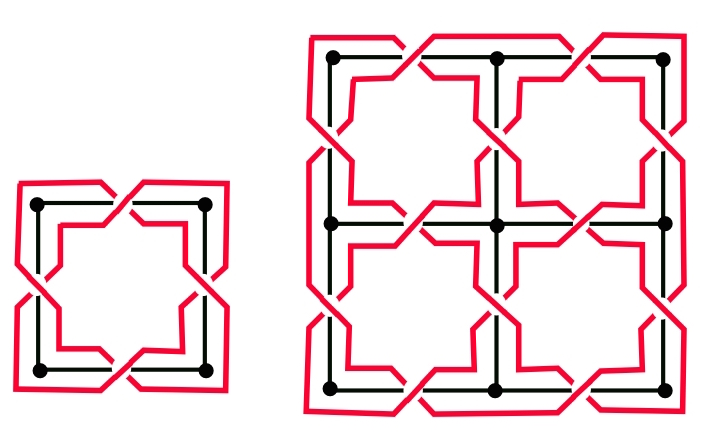}
\caption{Graphs $(\Gr_2)|_R$ and associated links, $\La = \langle x_1^2, x_2^2 \rangle$ and $\langle x_1^3, x_2^3 \rangle$}
\label{gridlinks}
\end{center}
\end{figure}

\end{example}

The {\it determinant} of a link $\ell$, denoted here by $d(\ell)$, is the absolute value of its 1-variable Alexander polynomial evaluated at $-1$.   We recall that a link $\ell$ is \emph{separable} if some embedded $2$-sphere in $\S^3 \setminus \ell$ bounds a $3$-ball containing a proper sublink of $\ell$. Otherwise $\ell$ is \emph{nonseparable}. Any link is the union of nonseparable sublinks.

The determinant of a separable link vanishes. We denote by $D(\ell_R)$ the nonzero product $d(\ell_1) \cdots d(\ell_r)$, where $\ell_1, \ldots, \ell_r$ are the nonseparable sublinks that comprise $\ell$. 

It follows from the Mayberry-Bott theorem \cite{BM54} that if $\ell$ is an alternating link that arises by the medial construction from a finite plane graph, then $d(\ell)$ is equal to the number of spanning trees of the graph (see appendix A.4 in\cite{BZ85}). 
The following corollary is an immediate consequence of Theorem \ref{limit}.  It has been proven independently by Champanerkar and Kofman \cite{CK16}.

\begin{cor}\label{detgrowth} Let $G$ be a plane graph with cofinite free $\Z^d$-symmetry, $d \le 2$. Then $$\lim_{\langle \La \rangle \to \infty}  \frac{1}{\vert \Z^d/\La\vert} \log D(\ell_R) = m(\De).$$ \end{cor}

\begin{remark} (1) As in Theorem \ref{limit}, if  each $G\vert_R$ is connected, then no link $\ell_R$ is separable. In this case, $D(\ell_R)$ is equal to the ordinary determinant of $\ell_R$. 

(2) In \cite{CKP15} the authors consider as well as more general sequences of links.  When $G = \Gr_2$, their results imply that:

$$\lim_{\langle \La \rangle \to \infty} \frac{2 \pi}{c(\ell_R)} \log d(\ell_R)=  v_{oct},$$
where $c(\ell_R)$ is the number of crossings of $\ell_R$ and $v_{oct} \approx 3.66386$ is the volume of the regular ideal octohedron.

\end{remark}

Grid graphs are the simplest connected locally finite graphs admitting free $\Z^d$-symmetry, as the 
following theorem shows. 

\begin{theorem} \label{min} If $G$ is a graph with cofinite free $\Z^d$-symmetry and finitely many connected components, then $\gamma(G)\ge \gamma(\Gr_d)$, and so
$m(\De(G)) \ge m(\De(\Gr_d))$. 
\end{theorem} 

\begin{remark} If $G$ has infinitely many connected components, then the conclusion of Theorem \ref{min} need not hold. Consider, for example, the graph $\Gr_2$ with every vertical edge deleted. The graph has cofinite free $\Z^2$-symmetry. It follows from Lemma \ref{new} below that its complexity growth rate is equal to $m(\De(\Gr_1))=0$, which is less than $m(\De(\Gr_2))$. 
\end{remark}

The following lemma, needed for the proof of Corollary \ref{absolute}, is of independent interest. 

\begin{lemma} \label{nonincreasing} The sequence of complexity growth rates $m(\De(\Gr_d))$ is nondecreasing. 
\end{lemma} 

Doubling each edge of $\Gr_1$ results in a graph with Laplacian polynomial $2(2-x-x^{-1})$, which has logarithmic Mahler measure $\log 2 + m(2-x-x^{-1}) = \log 2$. The following corollary states that this is minimum nonzero complexity growth rate. 

\begin{cor} \label{absolute} (Complexity Growth Rate Gap) Let $G$ be any graph with cofinite free $\Z^d$-symmetry and Laplacian polynomial $\De$. If $m(\De) \ne 0$, then 
$$m(\De) \ge \log 2.$$
\end{cor}

Although $\De(\Gr_d)$ is relatively simple, the task of computing its Mahler measure is not. It is well known and not difficult to see that $m(\De(\Gr_d)) \le \log 2d$. We will use Alon's result (Theorem \ref{lowerbound}) to show that $m(\Gr_d)$ approaches $\log 2d$ asymptotically. 

\begin{theorem} \label{gridlim} (1) For every $d \ge 1$, $m(\De(\Gr_d)) \le \log 2d$.\\ 
(2) $\lim_{d \to \infty} m(\De(\Gr_d)) - \log 2d = 0.$ \end{theorem}

Asymptotic results about the Mahler measure of certain families of polynomials have been obtained elsewhere. However, the graph theoretic methods that we employ to prove Theorem \ref{min} are different from techniques used previously.

\section{Algebraic dynamical systems and proofs.}\label{proofs}

We review some of the ideas of algebraic dynamical systems found in \cite{LSW90} and \cite{Sc95}.

For any finitely generated module $M$ over $\Rd$, we can consider the Pontryagin dual $\widehat M= {\rm Hom}(M, \T)$, where $\T$ is the additive circle group $\R/\Z$. We regard $M$ as a discrete space. Endowed with the compact-open topology, $\widehat M$ is a compact $0$-dimensional space. Moreover, the module actions of $x_1, \ldots, x_d$ determine commuting homeomorphisms $\si_1, \ldots, \si_d$ of $\widehat M$. Explicitly, $(\si_j \rho)(a) = \rho(x_j a)$ for every $a \in M$. Consequently, $\widehat M$ has a $\Z^d$-action $\si: \Z^d \to \text{Aut}(\widehat M)$. We will regard monomials $x^\s$ as acting on $\widehat M$ by $\si(\s)$ .

The pair $(\widehat M, \si)$ is an \emph{algebraic dynamical system}. It is well defined up to \emph{topological conjugacy}; that is, up to a homeomorphism of $\widehat M$ respecting the $\Z^d$ action. In particular its periodic point structure is well defined.  

\emph{Topological entropy} $h(\si)$ is another well-defined quantity associated to $(\widehat M, \si)$. (See \cite{LSW90} or \cite{Sc95} for the definition.) When $M$ can be presented by a square matrix $A$ with entries in $\Rd$, topological entropy can be computed as the logarithmic Mahler measure $m(\det A)$.

For any subgroup $\La$ of $\Z^d$, a $\La$-\emph{periodic point} is an element that is fixed by every $x^\s \in \La$. The set of all $\La$-periodic points is denoted by ${\rm Per}_\La(\si)$. It is a finitely generated abelian group isomorphic to ${\rm Hom}(T(M/\La M), \T)$, the Pontryagin dual of the torsion subgroup of $M/\La M$. The group consists of $|T(M/\La M)|$ tori of dimension equal to the rank of $M/\La M$.

We apply the above ideas to graphs $G$ with cofinite free $\Z^d$-symmetry. As in \cite{LSW14}, define the \emph{coloring module} $C$ to be the finitely presented module over the ring $\Rd$ with presentation matrix equal to the $n \times n$ Laplacian matrix $L$ of $G$. The Laplacian polynomial $\De$ arises as the $0$th elementary divisor of $C$.

Let $\La$ be a finite-index subgroup of $\Z^d$, and consider the $r$-sheeted covering graph $G_\La$. It has finitely many connected components. We denote by $n_\La$ the product of the cardinality of the vertex sets of the components. If $G$ is connected, then $n_\La= |\o V||\Z^d/\La|$. 

As in \cite{Sc95}, let $$\Omega(\La) = \{{\bf c}= (c_1, \ldots, c_d) \in \Co^d \mid {\bf c}^{\bf n} = 1\ \forall\ {\bf n} =(n_1, \ldots, n_d)\in \La\}.$$

The following combinatorial formula for the complexity $\tau(G_\La)$ is motivated by \cite{Kr78}. It is similar to the formula on page 621 of \cite{LSW90} and also page 191 of \cite{Sc95}. The proof here is relatively elementary. 

\begin{prop} \label{productformula} Let $G$ be a graph with cofinite free $\Z^d$-symmetry.  Let $\La$ be a finite-index subgroup of $\Z^d$. If $\De$ is the Laplacian polynomial of $G$, then \begin{equation}\label{formula} T(G_\La) =\frac{1}{n_\La} \prod_{    {(c_1, \ldots, c_d) \in \Omega(\La)\cap \S^d}\atop {\De(c_1, \ldots, c_d) \ne0}   } |\De(c_1, \ldots, c_d)|.
\end{equation}

\end{prop}

\begin{proof} Since $\La$ has finite index in $\Z^d$,  there exist positive integers $r_1, \ldots, r_d$ such that  $\Z^d/\La \cong \Z/(r_1) \oplus \cdots \oplus \Z/(r_d)$. We can choose a basis $u_1, \ldots, u_d$ for $\Z^d$ such that the coset of $u_i$ generates $\Z/(r_i)$ (Theorem VI.4 of \cite{Ro95} can be used). Let $\De' = \De(u_1, \cdots, u_d)$. Equation \ref{formula} becomes:

 \begin{equation}\label{formula2} T(G_\La) = \frac{1}{n_\La} \prod_{    {(\omega_1^{r_1}, \ldots, \omega_d^{r_d})=(1, \ldots, 1)}\atop {\De'(\omega_1, \ldots, \omega_d) \ne0}   } |\De'(\omega_1, \ldots, \omega_d)|.
\end{equation}

Let $P_{r_i}$ denote the $r_i \times r_i$ permutation matrix corresponding to the 
cycle $(1, 2, \ldots, r_i)$. With respect to the basis $u_1, \ldots, u_d$, the Laplacian matrix $L_\La$ for $G_\La$ can be obtained from the Laplacian matrix $L$ for $G$ by replacing 
each variable $u_i$ with the $r\times r$ tensor (Kronecker) product $U_i = I_1 \otimes \cdots I_{i-1} \otimes P_{r_i}\otimes I_{i+1} \otimes \cdots I_d$. Here $I_1, \cdots, I_d$ denote identity matrices of sizes $r_1 \times r_1, \ldots, r_d \times r_d$, respectively. Any scalar $c$ is replaced with $c$ times the $r \times r$ identity matrix. We regard  $L_\La$ as a block matrix with blocks of size $r \times r$. 

By elementary properties of tensor product, the matrices $U_i$ commute. Hence the blocks of the characteristic matrix $\l I - L_\La$ commute. The main result of \cite{KSW99} implies that the determinant of $\l I- L_\La$ can be computed by treating the blocks as entries in a $d \times d$ matrix, computing the determinant, which is a single $r \times r$ matrix $D$, and finally computing the determinant of $D$. 

The matrix $D$ is simply the Laplacian polynomial $\De'(U_1, \ldots, U_d)$. The matrices $U_i$ can be simultaneously diagonalized. For each $i$, let  $v_{i, 1}, \ldots, v_{i, r_i}$ be a basis of eigenvectors for 
$P_i$ with corresponding eigenvalues the $r_i$th roots of unity. Then 
$\{v_{1, j_1} \otimes \cdots \otimes v_{d, j_d} \mid 0 \le j_i < r_i\}\subset \Co^d$ is a basis of eigenvectors for $D$. With respect to such a basis,  $D$ is a diagonal matrix with diagonal entries  $\De'(\w_1, \ldots, \w_d)$, where $\w_i$ is any $r_i$th root of unity.   Using Corollary \ref{product} and changing variables back, the proof is complete. 

\end{proof}


\ni {\it Proof of Theorem \ref{limit}.} We must show that 
$$ \lim_{\langle \La \rangle \to \infty} \frac{1}{|\Z^d/\La|} \log T(G_\La)$$  exists and is equal to $m(\De)$ where $\De$ is the Laplacian polynomial of $G$. Consider the formula (\ref{formula}) for $T(G)$ given by Proposition \ref{productformula}. We will prove shortly that  
$$\lim_{\langle \La \rangle \to \infty} \frac{1}{|\Z^d/\La|} \log n_\La =0.$$ 
Assuming this, it suffices  to show that \begin{equation}\label{rsum} \lim_{\langle \La \rangle \to \infty} \frac{1}{|\Z^d/\La|} \log \prod |\De(c_1, \ldots, c_d)|= \lim_{\langle \La \rangle \to \infty} \frac{1}{|\Z^d/\La|} \sum \log |\De(c_1, \ldots, c_d)|= m(\De).\end{equation} Here the product and sum are over all  $d$-tuples $(c_1, \ldots, c_d)\in \Omega(\La)\cap \S^d$ such that $\De(c_1, \ldots, c_d) \ne 0$. By a unimodular change of basis, as in the proof of Proposition \ref{productformula}, we see that the second expression in (\ref{rsum}) is a Riemann sum for $m(\De)$. The contribution of vanishingly small members of the partition that contain zeros of $\De$ can be made arbitrarily small (see pages 58--59 of \cite{EW99}).   Hence the Riemann sums converge to $m(\De)$. 

It remains to show that $\lim_{\langle \La \rangle \to \infty} \frac{1}{|\Z^d/\La|} n_\La =0.$ For this it suffices to assume that $G$ is the $\Z^d$ orbit of a single, unbounded component. Then $G_\La$ is also the orbit of a single component $G_0$. It is stabilized by some nonzero element $w \in \Rd$. The cardinality $|V(G_0)|$ is at least as large as the cardinality of the orbit of the identity in $\Z^d/\La$ under translation by $w$. The line through the origin in the direction of $w$ intersects the fundamental region of $\Z^d/\La$ in a segment of length at least as $\langle \La \rangle$. Hence the cardinality of the orbit of the origin under $w$ is at least $\langle \La \rangle/|w|$. From this we conclude that $$|V(G_0)| \ge \frac{\langle \La \rangle}{|w|}.$$

To complete the argument, let $N= |\o V||\Z^d/\La|$  denote the number of vertices in $G_\La$. Let $k$ be the number of connected components of $G_\La$. Since the components are graph isomorphic (by the induced $\Z^d$ action), $n_\La$ is equal to $(N/k)^k$. Now
$$\lim_{\langle \La \rangle \to \infty} \frac{1}{|\Z^d/\La|} \log n_\La =\lim_{\langle \La \rangle \to \infty} \frac{1}{|\Z^d/\La|} \log \bigg(\frac{N}{k} \bigg)^k.$$ Letting $s= N/k$, the number of vertices in each component, we have
$$\lim_{\langle \La \rangle \to \infty} \frac{1}{|\Z^d/\La|} \log (s)^{\frac{N}{s}}= |\o V| \lim_{\langle \La \rangle \to \infty}  \frac{\log s}{s}.$$ The last limit is zero since $s$ must grow without bound. 
\qed \\

Now suppose $H$ is a subgraph of $G$ consisting of one or more connected components of $G$, such that the orbit of $H$ under $\Z^d$ is all of $G$.  Let 
$\Gamma < \Z^d$ be the stabilizer of $H$.  Then $\Gamma\cong \Z^{d'}$ for some $d'\le d$, and its action on $H$ can be regarded as a cofinite free action of $\Z^{d'}$.   Its complexity growth rate is given by 
$$\g(H)=\lim_{\<\La\>\to\infty} \frac{1}{|\Gamma/\La|} \log T(H_\La)$$
where $\La$ ranges over finite-index subgroups of $\Gamma$.

\begin{lemma}\label{new} Under the above conditions we have $\g(G)=\g(H)$.
\end{lemma}

\begin{proof}
Let $\La$ be any finite-index subgroup of $\Z^d$.  Then $H$ is invariant under $\La\cap\G$. The image of $H$ in the quotient graph $G_\La$ is isomorphic to  $H_{\La\cap\G}$.  

Note that the quotient $\overline H$ of $H$ by the action of $\G$ is isomorphic to $\overline G$, since the $\Z^d$ orbit of $H$ is all of $G$.  Since $G_\La$ is a $|\Z^d /\La|$-fold cover of $\overline G$ and $H_{\La\cap\G}$ is a $|\G/(\La\cap\G)|$-fold cover of $\overline H$, $G_\La$ comprises $k=|\Z^d /\La| / |\G/(\La\cap\G)|$ mutually disjoint translates of a graph that is isomorphic to $H_{\La\cap\G}$.  Hence $T(G_\La)=T(H_{\La\cap\G})^k$ and
$$\frac{1}{|\Z^d /\La|} \log T(G_\La)= \frac{1} {|\G/(\La\cap\G)|} \log T(H_{\La\cap\G}).$$
Since $\<\La\cap\G\>\to\infty$ as $\<\La\>\to\infty$, we have $\g(G)=\g(H)$.
\end{proof}

\noindent {\it Proof of Theorem \ref{min}.}  
By Proposition \ref{components}, we may assume that $G$ is the orbit of a single connected component $H$.  Since $G$ has finitely many components, the stabilizer $\G$ of $H$ is isomorphic to $\Z^d$ and has a cofinite free action on $H$, with $\g(G)=\g(H)$ by Lemma \ref{new}.  Thus we can assume $G$ is connected.

Consider the case in which $G$ has a single vertex orbit. Then for some $u_1,\ldots, u_m\in \Z^d$, the edge set $E$ consists of edges from $v$  to $u_i \cdot v$ for each $v\in V$ and $i=1,\ldots,m$.  Since $G$ is connected, we can assume after relabeling that $u_1, \ldots, u_d$ generate a finite-index subgroup of $\Z^d$. Let $G'$ be the be the $\Z^d$-invariant subgraph of $G$ with edges from $v$  to $u_i \cdot v$ for each $v\in V$ and $i=1,\ldots,d$.  Then $G'$ is the orbit of a subgraph of $G$ that is isomorphic to $\Gr_d$, and so by Lemma \ref{new}, $\g(\Gr_d)=\g(G')\le\g(G)$.

We now consider a connected graph $G$ having vertex families $v_{1, \s}, \ldots, v_{n, \s}$, where $n >1$. Since $G$ is connected, there exists an edge $e$ joining $v_{1, {\bf 0}}$ to some $v_{2, \s}$. Contract the edge orbit $\Z^d \cdot e$ to obtain a new graph $G'$ having cofinite free $\Z^d$-symmetry and complexity growth rate no greater than that of $G$.
Repeat the procedure with the remaining vertex families so that only $v_{1, \s}$ remains. The proof in the previous case of a graph with a single vertex orbit now applies. \qed \\


\noindent {\it Proof of Lemma \ref{nonincreasing}.}  Consider the grid graph $\Gr_d$. Deleting all edges 
in parallel to the $d$th coordinate axis yields a subgraph $G$ consisting of countably many mutually disjoint translates of $\Gr_{d-1}$. By Lemma \ref{new},   $m(\De(\Gr_{d-1}))=m(\De(G))  \le m(\De(\Gr_d))$. 
\qed \\

\noindent {\it Proof of Corollary \ref{absolute}.}  By Proposition \ref{components} and Lemma \ref{new}, it suffices to consider a connected graph $G$ with cofinite free $\Z^d$-symmetry and $m(\De(G))$ nonzero. Note that $m(\Gr_1) =0$ while $m(\Gr_2) \approx 1.166$ is greater than $\log 2$.  By Theorem \ref{min} and Lemma \ref{nonincreasing} we can  assume that $d=1$. 

If $G$ has an orbit of parallel edges, we see easily that $\g(G)\ge \log 2$.  Otherwise, we proceed as in the proof of  Theorem \ref{min}, contracting edge orbits to reduce the number of vertex orbits  without increasing the complexity growth rate.  If at any step we obtain an orbit of parallel edges, we are done; otherwise we will obtain a graph $G'$ with a single vertex orbit and no loops. 
If $G'$ is isomorphic to $\Gr_1$ then $G$ must be a tree; but then $m(\De(G))=\g(G)=0$, contrary to our hypothesis. So $G'$ must have at least two edge orbits.  
Deleting excess edges, we may suppose $G'$ has exactly two edge orbits. 

The Laplacian polynomial $m(\De(G'))$ has the form $4-x^r-x^{-r}- x^s-x^{-s}$, for some positive integers $r, s$. Reordering the vertex set of $G'$, we can assume without loss of generality that $r=1$. The following calculation is based on an idea suggested to us by Matilde Lalin.

$$m(\De(G')) = \int_0^1 \log \vert 4- 2 \cos( 2 \pi \theta)-2 \cos(2 \pi s \theta) \vert \ d\theta$$

$$=\int_0^1 \log \vert  2(1-\cos  (2 \pi \theta)) + 2(1 - \cos(2 \pi s \theta)) \vert \ d\theta$$

$$=\int_0^1 \log  \bigg( 4\sin^2(\pi \theta) + 4\sin^2(\pi s \theta) \bigg) \ d\theta.$$

\ni Using the inequality $(u^2+v^2) \ge 2 u v$, for any nonnegative $u, v$, we have:

$$m(\De(G')) \ge \int_0^1 \log \bigg( 8 \vert \sin( \pi \theta)\vert\  \vert \sin( \pi s \theta) \vert \bigg)\ d\theta$$

$$= \log 8 + \int_0^1 \log \vert \sin(  \pi \theta) \vert \ d\theta + \int_0^1 \log \vert \sin( \pi s \theta) \vert \ d\theta$$

$$= \log 8 + \int_0^1 \log \sqrt{\frac{1-\cos(2 \pi \theta)}{2}}\ d \theta + \int_0^1 \log \sqrt{\frac{1-\cos(2  \pi s\theta)}{2}}\ d \theta$$

$$=\log 8 + \int_0^1 \frac{1}{2} \log \bigg(\frac{2 - 2 \cos(2 \pi \theta)}{4}\bigg)\ d \theta +   \int_0^1 \frac{1}{2} \log \bigg(\frac{2 - 2 \cos(2 \pi s \theta)}{4}\bigg)\ d \theta$$

$$= \log 8 + \frac{1}{2}m(2 - x -x^{-1}) -\frac{1}{2} \log 4 + \frac{1}{2}m(2 - x^s - x^{-s})- \frac{1}{2} \log 4$$

$$= 3\log 2 + 0 - \log 2 +0 -\log 2 = \log 2.$$
\qed\\

\noindent{\it Proof of Theorem \ref{gridlim}.} (1) The integral representing the logarithmic Mahler measure of $\De(\Gr_d)$ can be written
$$\int_0^1 \cdots \int_0^1 \log \bigg\vert2d - \sum_{i=1}^d 2 \cos(2 \pi \theta_i)\bigg\vert d\theta_1\cdots d\theta_d$$
$$= \log 2d + \int_0^1 \cdots \int_0^1 \log \bigg\vert1+\sum_{i=1}^d \frac{ \cos( 2 \pi \theta_i)}{d}\bigg\vert d\theta_1\cdots d\theta_d$$
$$= \log 2d + \int_0^1 \cdots \int_0^1 -\sum_{k=1}^\infty \frac{(-1)^k}{k}\bigg( \frac{\sum_{i=1}^d \cos(2 \pi \theta_i)}{d} \bigg)^k d\theta_1\cdots d\theta_d.$$
By symmetry, odd powers of $k$ in the summation contribute zero to the integration. Hence 
$$m(\De(\Gr_d)= \log 2d - \int_0^1 \cdots \int_0^1 \sum_{k=1}^\infty \frac{1}{2k}\bigg( \frac{\sum_{i=1}^d \cos(2 \pi \theta_i)}{d} \bigg)^{2k} d\theta_1\cdots d\theta_d \le \log 2d.$$

(2) Let $\La$ be a finite-index subgroup of $\Z^d$. Consider the quotient graph $(\Gr_d)_\La$. The cardinality of its vertex set is $|\Z^d/\La|$. The main result of \cite{Al90}, cited above as Theorem \ref{lowerbound}, implies that 
$$\t((\Gr_d)_\La) = \bigg((2 d) (1 - \mu(d))\bigg)^{|\Z^d/\La|},$$
where $\mu$ is a nonnegative function such that $\lim_{d\to \infty} \mu(d) =0$.  
Hence 
$$\lim_{d \to \infty}  \bigg (\frac{1}{|\Z^d/\La|} \log \t((\Gr_d)_\La - \log 2 d\bigg ) = \lim_{d\to \infty} \log (1-\mu(d)) =0.$$
Theorem \ref{limit} completes the proof.  
\qed \\

\begin{remark} One can evaluate $m(\De(\Gr_d))$ numerically and obtain an infinite series representing
$m(\De(\Gr_d)) - \log 2d$. However, showing rigorously that the sum of the series approaches zero as $d$ goes 
to infinity appears to be difficult. (See \cite{SW00}, p. 16 for a heuristic argument.) 
\end{remark}

\bigskip

\ni Department of Mathematics and Statistics,\\
\ni University of South Alabama\\ Mobile, AL 36688 USA\\
\ni Email: \\
\ni  silver@southalabama.edu\\
\ni swilliam@southalabama.edu
\end{document}